\documentclass[12pt]{article}
\usepackage {graphics} 
\usepackage {graphicx}  
\usepackage {pstricks,pst-node,  pst-coil} 
\onecolumn  
\usepackage{amsmath,amssymb,latexsym}
\usepackage{amsmath,amsthm}
\usepackage{subcaption}

\usepackage{tikz}

\pgfdeclarelayer{edgelayer}
\pgfdeclarelayer{nodelayer}
\pgfsetlayers{edgelayer,nodelayer,main}

\tikzstyle{black_dot}=[draw=black, shape=circle, inner sep=2pt, minimum size=15pt]
\tikzstyle{white_circle}=[fill=white, draw=black, shape=circle, inner sep=0pt, minimum size=3pt]
\tikzstyle{empty-vertex}=[fill=none, draw=none, shape=rectangle, minimum size=0pt, inner sep=0pt]

\tikzstyle{dashed_edge}=[-, dashed, thick]
\tikzstyle{blue_edge}=[-, draw=blue, thick]
\tikzstyle{red_edge}=[-, draw=red, thick]
\tikzstyle{blue_dashed}=[-, draw=blue, dashed, thick]
\tikzstyle{red_dashed}=[-, draw=red, dashed, thick]
\tikzstyle{blue_arc}=[->, draw=blue, thick]
\tikzstyle{red_arc}=[->, draw=red, thick]
\tikzstyle{green_edge}=[-, draw=green, thick]
\tikzstyle{yellow_edge}=[-, draw=yellow, thick]
\tikzstyle{arrow}=[fill=none, ->, thick]

\newtheorem{tw}{Theorem}  
\newtheorem{lem}[tw]{Lemma}  
  
\newtheorem{cnj}[tw]{Conjecture}  
\newtheorem{cor}[tw]{Corollary}  
  
\newtheorem{obs}[tw]{Observation} 
\newtheorem{claim}{Claim}
\newtheorem{problem}{Problem}

\begin{document}
\hyphenation{every}
\title{Weak and strong local irregularity of digraphs}
\author{Igor Grzelec\thanks{Department of Discrete Mathematics, AGH University of Krakow, Poland} \thanks{The corresponding author. Email:  grzelec@agh.edu.pl}, Alfréd Onderko\thanks{Institute of mathematics, P.J. Šafárik University, Košice, Slovakia}, Mariusz Woźniak\footnotemark[1]}

\maketitle
\begin{abstract}
    Local Irregularity Conjecture states that every simple connected graph, except special cacti, can be decomposed into at most three locally irregular graphs, i.e., graphs in which adjacent vertices have different degrees.
    The connected minimization problem, finding the minimum number $k$ such that a graph can be decomposed into $k$ locally irregular graphs, is known to be NP-hard in general (Baudon, Bensmail, and Sopena, 2015).
    This naturally raises interest in the study of related problems.
    Among others, the concept of local irregularity was defined for digraphs in several different ways.
    In this paper we present the following new methods of defining a locally irregular digraph.
    The first one, weak local irregularity, is based on distinguishing adjacent vertices by indegree-outdegree pairs, 
    and the second one, strong local irregularity, asks for different balanced degrees (i.e., difference between the outdegree and the indegree of a vertex) of adjacent vertices.
    For both of these irregularities, we define locally irregular decompositions and colorings of digraphs.
    We discuss relation of these concept to others, which were studied previously, and provide related conjectures on the minimum number of colors in weak and strong locally irregular colorings.
    We support these conjectures with new results, using the chromatic and structural properties of digraphs and their skeletons (Eulerian and symmetric digraphs, orientations of regular graphs, cacti, etc.).  

\end{abstract}

\section{Introduction}
In this work, we investigate generalizations of locally irregular graphs and locally irregular decompositions from (simple undirected) graphs to digraphs.
First, we explain the notion of locally irregular graphs and present some results of decomposing graphs into locally irregular graphs.

We call a graph $G$ \textit{locally irregular} if every two neighboring vertices of $G$ have different degrees. 
Decompositions of graphs into locally irregular subgraphs are often described in the language of colorings. 
We define \textit{locally irregular coloring} of a graph $G$ as an edge coloring of $G$ such that every color class induces a locally irregular subgraph of $G$. 
The \textit{locally irregular chromatic index} of a graph  $G$, denoted by $\mathrm{lir}(G)$, is the minimum number of colors required in the locally irregular coloring of $G$ if such a coloring exists. 

The problem of determining the value of $\mathrm{lir}(G)$ is closely related to the well-known 1-2-3 Conjecture stated by Karoński, Łuczak and Thomason in \cite{Karonski Luczak Thomason}. In 2024 the conjecture was confirmed by Keusch in \cite{Keusch}, where he proved that for every graph containing no isolated edges exists a \textit{neighbor-sum-distinguishing} edge coloring (i.e. an edge coloring in which every two neighboring vertices have different sums of colors of incident edges) using at most three colors. A weaker version of neighbor-sum-distinguishing edge coloring is  \textit{multiset neighbor distinguishing} edge coloring, where every two neighboring vertices have different multisets of colors of incident edges, because if the sums of colors are different then the multisets are also different. The best known result about the multiset neighbor distinguishing edge coloring stats that for every graph containing no isolated edges exists such a coloring using at most three colors and was proved by Vu\v cković in \cite{Vuckovic}. Obviously Keusch result about neighbor-sum-distinguishing edge coloring imply Vu\v cković theorem about multiset neighbor distinguishing edge coloring. On the other hand, if the graph $G$ satisfies ${\rm lir}(G)\leq k$, then it has multiset neighbor distinguishing edge coloring with at most $k$ colors. This holds because in a decomposition of the graph $G$ into $k$ locally irregular graphs, every two neighboring vertices in $G$ have different degrees in at least one locally irregular graph. Therefore every two neighboring vertices in $G$ have multisets differing in the multiplicity of at least one element. 

We can easily see that not every graph has a locally irregular coloring; easy examples of such graphs are odd-length paths and cycles.
In \cite{Baudon Bensmail Przybylo Wozniak} Baudon, Bensmail, Przybyło, and Woźniak fully characterized all non locally irregular colorable graphs. 
They also proposed the Local Irregularity Conjecture which states that every locally irregular colorable graph $G$ satisfies $\mathrm{lir}(G)\leq 3$.

Although both the $1$-$2$-$3$ Conjecture and the Local Irregularity Conjecture imply the fact that one can distinguish neighboring vertices by multisets using three colors, these properties are independent of each other. In the situation when the $1$-$2$-$3$ Conjecture was proved, the conjecture that $\mathrm{lir}(G)\leq 3$ for every locally irregular colorable graph $G$ remains the main focus of research on this topic. As we will see, we will encounter similar situations in case of digraphs.

Due to the fact that all non locally irregular colorable graphs are subclass of cacti (connected graphs with edge-disjoint cycles) 
determining the value of locally irregular chromatic index of cacti was set as a goal by Sedlar and Škrekovski in~\cite{Sedlar Skrekovski 2} and~\cite{Sedlar Skrekovski}. During their research, a single counterexample to the Local Irregularity Conjecture was found in~\cite{Sedlar Skrekovski}. 
In~\cite{Sedlar Skrekovski 2}, they proved that this counterexample (whose value of the locally irregular coloring index equals four), the bow-tie graph $B$ (see Figure \ref{graph_B}), is the only counterexample to the Local Irregularity Conjecture among all colorable cacti. 
    \begin{figure}[h!]
    \centering
    \includegraphics[width=5cm]{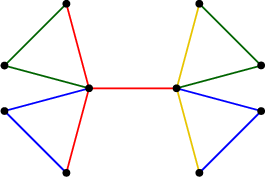}
    \caption{Locally irregular coloring of the bow-tie graph $B$ using four colors.}
    \label{graph_B}
    \end{figure}
In light of these new results, the Local Irregularity  Conjecture was improved in the following way.

\begin{cnj}[Local Irregularity Conjecture~\cite{Baudon Bensmail Przybylo Wozniak}, \cite{Sedlar Skrekovski 2}]
\label{graph3}
    Let $G \neq B$ be a connected locally irregular colorable graph. Then $\mathrm{lir}(G)\leq 3$.
\end{cnj}

The Local Irregularity Conjecture was confirmed for some graph classes;
among others trees~\cite{Baudon Bensmail Przybylo Wozniak}, complete graphs~\cite{Baudon Bensmail Przybylo Wozniak}, graphs with minimum degree at least $10^{10}$~\cite{Przybylo}, and $r$-regular graphs, where $r\geq10^6$~\cite{Baudon Bensmail Przybylo Wozniak}. 
Additionally, for every locally irregular colorable graph $G$, Bensmail, Merker and Thomassen~\cite{Bensmail Merker Thomassen} proved that $\mathrm{lir}(G)\leq 328$.
Later Lu\v zar, Przybyło, and Soták decreased this bound to 220 in~\cite{Luzar Przybylo Sotak}.

There are several ways to generalize the concept of locally irregular decompositions in terms of oriented graphs and digraphs. 
Note that, in this paper, we distinguish between oriented graphs (orientations of simple graphs) and digraphs; 
in contrast to oriented graphs, arcs going in opposite directions between a pair of vertices are allowed in digraphs. 
For a vertex $v$ of a digraph, we denote by $d^+(v)$ the outdegree (i.e. the number of outgoing arcs from $v$) and by $d^-(v)$ the indegree (i.e. the number of incoming arcs to $v$) of $v$. 

Now we introduce the notation that we use to distinguish between previously studied concepts of local irregularity in digraphs.
For $a,b \in \{+,-\}$ we say that a digraph $D$ is \textit{$(a,b)$-locally irregular} if for every arc $uv \in E(D)$ we have $d^a(u) \neq d^b(v)$.
For example (for $a =  +$ and $b = -$), we say that a digraph $D$ is $(+,-)$-locally irregular if the outdegree of $u$ is different from the indegree of $v$ for each arc $uv$ of $D$.
The definition of an $(a,b)$-locally irregular digraph, for $a,b \in \{+,-\}$, allows us to naturally define an \textit{$(a,b)$-locally irregular coloring} as an arc coloring in which each color class induce an $(a,b)$-locally irregular digraph.
The \textit{$(a,b)$-locally irregular chromatic index} of $D$, denoted by $\mathrm{lir}^{(a,b)}(D)$, is then the minimum number of colors in an $(a,b)$-locally irregular coloring of $D$.

It is easy to see that $(-,-)$-locally irregular coloring and $(+,+)$-locally irregular coloring of a digraph $D$ are analogous notions because a $(-,-)$-irregular decomposition of $D$ gives us a $(+,+)$-irregular decomposition of the digraph obtained from $D$ by reversing the orientation of every arc, and vice versa. 
Hence, all mentioned results on $(+,+)$-locally irregular colorings may also be viewed as results on $(-,-)$-locally irregular colorings. 

The $(+,+)$-locally irregular colorings were inspired by neighbor-sum-distinguishing arc colorings introduced by Baudon, Bensmail and Sopena in \cite{Baudon Bensmail Sopena}. We say that an arc coloring is \textit{neighbor-sum-distinguishing} if every two neighboring vertices have different sums of colors of its outgoing arcs. The authors in \cite{Baudon Bensmail Sopena} proved that every oriented graph $D$ has neighbor-sum-distinguishing arc coloring with at most three colors.

Bensmail and Renault in \cite{Bensmail Renault} observed that there are no digraphs which are non $(+,+)$-locally irregularly colorable and proposed the following conjecture.
\begin{cnj}[\cite{Bensmail Renault}]
    Every oriented graph $D$ satisfies $\mathrm{lir}^{(+,+)}(D)\leq 3$.
\end{cnj}
Moreover, they proved that deciding whether $\mathrm{lir}^{(+,+)}(D)\leq2$ holds for a given oriented graph $D$ is NP-complete and showed that $\mathrm{lir}^{(+,+)}(D)\leq6$ for every oriented graph $D$. 
In~\cite{Baudon Bensmail Przybylo Wozniak2} the upper bound on $\mathrm{lir}^{(+,+)}(D)$ was lowered to five for arbitrary digraph $D$.

Recently, the $(+,-)$-locally irregular coloring of digraphs was introduced by Bensmail et al. in \cite{Bensmail Filasto Hocquard Marcille}. The inspiration of that locally irregular coloring of digraphs was Barme et al. result from \cite{Barme Bensmail Przybylo Wozniak}, which says that every digraph $D$ without arc $uv$ verifying $d^+(u)=d^-(v)=1$ has arc coloring using colors from the set $\{1,2,3\}$ so that for every arc $uv$ of $D$, the sum of colors of outgoing arcs from $u$ is different from the sum of colors of ingoing arcs to $v$.
Bensmail et al. observed that only digraphs which contain (an oriented) sink-source path do not have $(+,-)$-locally irregular coloring and proposed the following conjecture.
\begin{cnj}[\cite{Bensmail Filasto Hocquard Marcille}]
    Every digraph $D$ without sink-source path satisfies $\mathrm{lir}^{(+,-)}(D)\leq 3$.
\end{cnj}
The authors also provided a constant upper bound equal to seven for the $(+,-)$-locally irregular chromatic index of any digraph that does not contain a sink-source path. 
Additionally, they made some remarks on $(-,+)$-locally irregular colorings of digraphs.

In this paper we present two new concepts of local irregularities of digraphs and the corresponding coloring problems. 
The first one is based on distinguishing neighbors by the outdegree-indegree pairs. 
Such a condition is a weakening of the condition that outdegrees (or indegrees) of two adjacent vertices are different (used in the definition of $(+,+)$-locally irregular decomposition and coloring). 
Hence, we use \textit{weak locally irregular} to denote such a decomposition and a coloring.
We propose a conjecture that 2 colors are enough for a weak local irregular decomposition of any digraph; this suggest that such a weakening  of $(+,+)$-irregularity results in needing fewer colors in the general case (see Conjecture 3.1 in~\cite{Baudon Bensmail Przybylo Wozniak2} that states that ${\rm lir}^{(+,+)}(D)\leq 3$ for each digraph $D$). In Section~\ref{section_weak}, we support the conjecture on weak locally irregular chromatic index with several new results.

For the second type of local irregularity that is investigated in this paper, we use balanced degrees as a distinguishing factor between adjacent vertices. 
A \textit{balanced degree} of a vertex $x$ in a digraph $D$, denoted by $\sigma(x)$, is the difference of the outdegree of $x$ and its indegree, i.e, $\sigma(x) = d^+(x) - d^-(x)$. The inspiration for using balanced degrees as a distinguishing factor between adjacent vertices comes from Borowiecki, Pilśniak and Grytczuk result, from \cite{Borowiecki Grytczuk Pilśniak}, which says that every digraph $D$ has an arc coloring with colors from the set $\{1,2\}$ so that every two adjacent vertices in $D$ has different values of the parameter defined for every vertex as the sum of colors of outgoing arcs minus sum of colors of ingoing arcs. Note that, for an arc $xy$, the condition $\sigma(x) \neq \sigma(y)$ is more restrictive than the condition $(d^+(x),d^-(x)) \neq (d^+(y),d^-(y))$: while the former one implies the later one, the converse is not true.
However, it is not true that the condition $\sigma(x) \neq \sigma(y)$ for every arc $xy$ in $D$ strengthens the $(+,+)$-locally irregular condition.
Hence, the corresponding coloring concepts, strong locally irregular colorings and $(+,+)$-locally irregular colorings are in general incomparable.
We investigate strong locally irregular colorings in Section~\ref{section_strong}, 
provide a conjecture on a strong locally irregular chromatic index of a digraph, and support it with new results. 

\section{Weak local irregularity} \label{section_weak}

In this section, we discuss the local irregularity in the sense that pairs of outdegrees and indegrees of adjacent vertices are different, i.e., if $xy \in E(D)$ then $(d^+(x), d^-(x)) \neq (d^+(y), d^-(y))$.
Locally irregular coloring and locally irregular chromatic index, denoted by ${\rm lir}(D)$, of a digraph $D$ are defined naturally.
Similarly to the case of $(+,+)$-locally irregular decompositions and colorings, a single arc is locally irregular. 
Therefore, assigning a unique color to each arc of a digraph yields a feasible coloring. 
Hence, every digraph admits a locally irregular coloring in this sense. 
Generally, an upper bound given by $|E(D)|$ can be far from being optimal. 
We state the following conjecture, supported by several results in this section (Theorems~\ref{dwudzielne}, \ref{trojdzielne}, \ref{tournament}, and~\ref{lir_symmetric}).

\begin{cnj}\label{main}
    Every digraph $D$ satisfies ${\rm lir}(D)\leq 2$.
\end{cnj}

Note that any orientation of a simple locally irregular graph yields a locally irregular oriented digraph: 
if the degrees of two adjacent vertices in a simple graph are different, then in an orientation of this graph, either outdegrees or indegrees of these two vertices are different. 
This, in particular, means that if Conjecture~\ref{graph3} is true, an orientation of every connected graph (except for special cacti), is colorable using at most 3 colors.

A \textit{skeleton} of a digraph $D$ is a simple graph obtained from $D$ by replacing every its arc with a single simple edge, and, additionally, replacing the multiedges with single edges.
Hence, the observation from the previous paragraph could be restated: if a skeleton of an oriented graph is locally irregular then the oriented graph is weak locally irregular.
For general digraphs, however, this is not true: consider for example a 5-vertex digraph $D$ given by $E(D) = \{v_2v_1,v_2v_3,v_4v_2,v_4v_5,v_5v_4\}$ and its skeleton $G$.
While $G$ is a locally irregular tree, $D$ is not weak locally irregular ($v_2$ and $v_4$ are adjacent and they have the same outdegree-indegree pairs).

Given the current state of the art, Conjecture~\ref{graph3} is far from being proved for general graphs. However, the connection between the weak local irregularity and $(+,+)$-local irregularity of digraphs (the later one is a special case of the former one) results in the fact that every upper bound on $\mathrm{lir}^{(+,+)}(D)$ upper bounds the weak locally irregular chromatic index of $D$. Hence, from the result of Baudon et al. \cite{Baudon Bensmail Przybylo Wozniak2}, which states that $\mathrm{lir}^{(+,+)}(D)\leq 5$ for every digraph $D$, we get:
\begin{tw}\label{theorem_lir_general_upper_bound}
   For every digraph $D$ we have ${\rm lir}(D)\leq 5$.
\end{tw}
In what follows, we improve this general upper bound for special classes of digraphs. The results are split into two separate subsections. The first one deals with classes of digraphs whose skeletons have chromatic number (the minimum number of colors in a proper vertex coloring of a graph) upper bounded by a constant. The second one then consists of results on digraphs from classes, where the chromatic number of a skeleton cannot be upper bounded by a constant.

\subsection{Digraphs with $k$-chromatic skeletons}

Here, we use the value of the chromatic number of the skeleton $G$ of a digraph $D$, and the natural partition of $V(G)$ into $\chi(G)$ independent sets, as a base for a coloring scheme, according to which arcs of $D$ can be colored to obtain a locally irregular coloring of $D$. 
We start with digraphs, whose skeletons are bipartite or tripartite, i.e., whose skeletons have chromatic numbers equal to 2 or 3; for such digraphs, we prove Conjecture~\ref{main}. 

\begin{tw}\label{dwudzielne}
    If the skeleton of a digraph $D$ is bipartite then ${\rm lir}(D)\leq 2$.
\end{tw}
\begin{proof}
    Let $X, Y$ be the bipartition of vertices of the skeleton of $D$. 
    Color each arc from $X$ to $Y$ red, and each arc from $Y$ to $X$ blue.
    Let $x \in X$ and $y \in Y$.
    If $xy \in E(D)$ then it is red, $x$ has positive red outdegree, and $y$ has its red outdegree equal to 0.
    If $yx \in E(D)$ then it is blue, $x$ has positive blue indegree, and $y$ has its blue indegree equal to 0.
    Therefore, the blue and the red subdigraphs are weak locally irregular, thus, the obtained coloring is feasible.
\end{proof}

\begin{tw}\label{trojdzielne}
    If the skeleton of a digraph $D$ is tripartite then ${\rm lir}(D)\leq 2$.
\end{tw}
\begin{proof}
    Consider the partition of the vertices of the skeleton of $D$ into independent sets $X$, $Y$, and $Z$. 
    Color every arc from $X$ to $Y$, from $X$ to $Z$, and from $Y$ to $Z$ red. 
    Then, color the remaining arcs, i.e., the arcs from $Y$ to $X$, from $Z$ to $X$, and from $Z$ to $Y$ blue. 
    We show that the red subdigraph is locally irregular (the analogous arguments hold for the blue subdigraph).

    Consider an edge $xy$ from $X$ to $Y$. 
    Clearly, in the red subdigraph of $D$, $d^-(x) = 0$ and $d^-(y) \geq 1$, hence the local irregularity condition holds for the arc $xy$. 
    Suppose that $xz$ is an arc from $X$ to $Z$; similarly to the previous case,   $d^-(x) = 0$ and $d^-(z) \geq 1$ in the red subdigraph of $D$. 
    If $yz$ is an arc from $Y$ to $Z$, $d^+(y) \geq 1$ and $d^+(z) = 0$ in the red subdigraph of $D$.
    Thus, the red subdigraph of $D$ is locally irregular. 
\end{proof}

From Theorem~\ref{dwudzielne}, Theorem~\ref{trojdzielne}, and basic facts about chromatic number of simple graphs, one can observe that Conjecture~\ref{main} holds for digraphs whose skeletons are paths, cycles, trees, unicyclic graphs, cacti (including the ones that are not locally irregular colorable), or outerplanar graphs.

Note that all digraphs whose skeleton is $K_4$ has locally irregular coloring with two colors. To prove this simple observation it suffices to consider two cases. First, the digraph is a tournament with four vertices. It will be a direct consequence of Theorem \ref{tournament}. Second, when a digraph on four vertices has at least one pair of arcs going both directions. In this case we create from a digraph $D$ a multigraph $M$ replacing all arcs by edges (arcs going both directions we replace by two parallel edges). Then we can easily check that each such multigraph $M$ on four vertices has locally irregular coloring with two colors, where multiedges can be colored both colors (in total we need to check 13 multigraphs). Obviously the same coloring of the initial digraph is also locally irregular. Since a digraph whose skeleton is $K_4$ can be easily colored with two colors such that the monochromatic subdigraphs are locally irregular and $K_4$ is the only subcubic graph with chromatic number greater than three, we get that Conjecture \ref{main} is also true for digraphs whose skeletons are subcubic.

Now, we generalize the method from the proof of Theorem \ref{dwudzielne} and Theorem \ref{trojdzielne} and prove that in the case of a digraph whose skeleton has chromatic number at most six, three colors are enough for the locally irregular coloring.
Note that the method, with the larger number of colors, can be used to prove the upper bound for a digraph whose skeleton is $k$-partite, for arbitrary $k$. 
However, the bound is, for larger $k$, further away from the conjectured value two, and it is larger than the general upper bound of five (see Theorem~\ref{theorem_lir_general_upper_bound}), which makes it unusable to prove any more useful results.

\begin{tw}\label{sześciodzielne}
Let $G$ be a skeleton of a digraph $D$. 
If $\chi(G) \leq 6$ then ${\rm lir}(D)\leq 3$.
\end{tw}
\begin{proof}
    Assume that $\chi(G) = 6$. 
    We will deal with the cases when $\chi(G) \in \{4,5\}$ later. 
    Denote by $X_1, \dots, X_6$ the partition of $V(G)$ into independent sets. 
    We use a coloring scheme depicted in Figure~\ref{fig2}; 
    this, for example, means that every arc from $X_1$ to $X_3$ is red, every arc from $X_3$ to $X_1$ is green, etc. 
    We show that such a coloring scheme yields three locally irregular monochromatic subdigraphs of $D$.
    
    \begin{figure}[h!]
    \centering
    \includegraphics[width=1\textwidth]{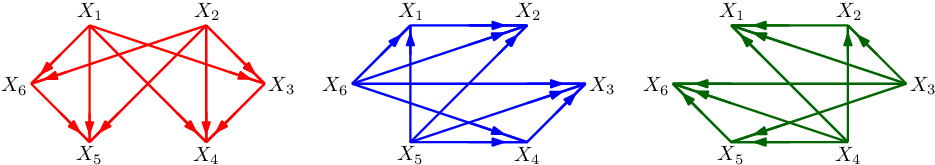}
    \caption{Scheme of a locally irregular decomposition of a digraph whose skeleton is 6-partite.}
    \label{fig2}
    \end{figure}
    
    In the case of the red subdigraph, call vertices from $X_1 \cup X_2$ poor, the vertices from $X_3 \cup X_6$ medium, and the vertices from $X_4 \cup X_5$ rich.
    Clearly, the indegree of every poor vertex is 0, while, on the other hand, the outdegree of every rich vertex is 0. 
    This is independent of how the original digraph $D$ looks. 
    Moreover, no arc in the red subdigraph connects two poor, two medium, or two rich vertices.
    It follows from these simple observations that the locally irregular condition holds for every arc $xy$ of the red subdigraph: if $x$ is rich and $y$ is medium or rich, then $d^-(x) = 0$ and $d^-(y) \geq 1$, and if $x$ is medium and $y$ is rich, then $d^+(x) \geq 1$ and $d^+(y) = 0$.
    
    Hence, the red subdigraph is always locally irregular. This is also true for the blue and green subdigraphs (simply change the roles of $X_1, \dots, X_6$ accordingly). 
    Thus, to prove that this gives a locally irregular decomposition of $D$, it remains to prove that each arc of $D$ is colored with one of the three possible colors. 
    This is easy to see since the coloring scheme in Figure~\ref{fig2} is a decomposition of the complete symmetric digraph on 6 vertices.

    If $\chi(G) \in \{4,5\}$, simply delete vertices $X_4$ and $X_5$, or $X_5$ from the digraphs in Figure~\ref{fig2}; the obtained digraphs define the coloring scheme which yields three locally irregular subdigraphs.
\end{proof}

From the above theorem and the well-known Four color theorem, we immediately get the following corollary.

\begin{cor}
    If the skeleton of a digraph $D$ is planar then ${\rm lir}(D)\leq 3$. 
\end{cor}

The following lemma is particularly useful to prove an upper bound for a $k$-partite graph, $k \in \{7,8,9\}$, using Theorem~\ref{trojdzielne}.

\begin{lem}
\label{3}
    Let $G$ be a simple graph with $\chi(G)=k$.
    Then $G$ can be written as the edge-disjoint union of $\lceil \log_p k \rceil$ $p$-partite graphs.
\end{lem}
\begin{proof}
    Let $l = \lceil \log_p k \rceil$. 
    Then $k\leq p^l$. 
    We show by induction on $l$ that $G$ is the union of $l$ $p$-partite graphs. 
    
    For $l = 1$ this is clear. 
    For $l > 1$, take a partition $V_1, \dots, V_k$ of vertices of $V(G)$ given by a proper vertex $k$-coloring of $G$.
    We divide these independent sets of vertices into $p$ groups, each having at most $\left \lceil k/p \right \rceil$ sets. Formally, $S_i = V_{(i-1)\left \lceil k/p \right \rceil + 1} \cup \dots \cup V_{i\left \lceil k/p \right \rceil}$, for every $i \in \{1, \dots, p-1\}$, and $S_p = V_{(p-1)\left \lceil k/p \right \rceil + 1} \cup \dots \cup V_k$.

    Let $G_1$ be the spanning $p$-partite graph formed by taking all edges of $G$ between vertices from $S_i$ and $S_j$, for $i,j \in \{1, \ldots, p\}$, $i \neq j$.
    Then $G \setminus E(G_1)$ has chromatic number at most  $\lceil \frac{k}{p}\rceil \leq p^{l-1}$. 
    Hence, by induction, $G \setminus E(G_1)$ can be written as the edge-disjoint union of $l-1$ $p$-partite graphs; denote them by $G_2, \ldots, G_l$. 
    Thus $G$ is the union of $l$ $p$-partite graphs $G_1, \dots, G_l$, as required.
\end{proof}

Using the previous lemma, we can write a simple $k$-partite graph, for $k \leq 9$, as an edge-disjoint union at most two 3-partite graphs. 
From Theorem ~\ref{trojdzielne}, the following theorem follows.

\begin{tw}
    If the skeleton of a digraph $D$ is $k$-partite for some $k \in \{7,8,9\}$ then ${\rm lir}(D)\leq 4$.
\end{tw}

\subsection{Classes of digraphs without a constant-bound on the chromatic number of a skeleton}

Now we present further results, in this case about the value of the locally irregular chromatic index of digraphs from special classes, in which the chromatic number of a skeleton is not upper-bounded by a constant. 
One such class of digraphs is a class of tournaments, i.e., orientations of complete graphs (arcs going in both directions between two vertices are not allowed). 
For tournaments, we prove that Conjecture~\ref{main} holds.

\begin{tw}\label{tournament}
    For every tournament $T_n$ on $n$ vertices we have ${\rm lir}(T_n)\leq 2$.
\end{tw}
\begin{proof}    
    Let $v_1$ be a vertex of $T_n$ such that $d^+(v_1) \geq d^-(v_1)$ (such a vertex exists since $\sum_{v \in V(T_n)}d^+(v) = \sum_{v \in V(T)} d^-(v)$).
    Thus, $d^+(v_1) \geq \left\lfloor (n-1)/2\right\rfloor$. 
    Denote the remaining vertices of $T_n$ by $v_2, \dots, v_n$ such that $v_1v_{2i}$ is an arc for each $i \in \{1, \dots, \left\lfloor n/2 \right\rfloor\}$.
        
    We now describe the coloring of the skeleton $K_n$ of $T_n$: every edge $v_{2i}v_j$ for $j < 2i$ is colored red, and every edge $v_{2i+1}v_j$ for $j < 2i+1$ is colored blue. 
    We claim that such an edge coloring is almost locally irregular coloring of $K_n$, i.e., there is only a single edge of $K_n$, namely $v_1v_2$, where the locally irregular property is violated.
    To see this, consider a coloring procedure where, in the $i$-th step, we color all arcs between $v_i$ and all the vertices $v_j$ where $j < i$.
    After the $i$-th step, colored edges induce $K_i$. 
    Suppose that after the coloring of $K_{i-1}$ right before the $i$-th step is almost locally irregular with the exception of $v_1v_2$ (this is clearly true for $i \leq 4$).
    Observe that if $i$ is even, the red degree of every vertex is at most $i-3$ (all edges incident to $v_{i-1}$ in $K_{i-1}$ are blue). 
    After adding $v_i$ and coloring all edges incident to $v_i$ red, $v_i$ has red degree $i-1$, and the remaining vertices have red degrees at most $i-2$. 
    Moreover, red degree of every vertex different from $v_i$ is raised by $1$, which means that the only edge for which the local irregular condition does not hold is $v_1v_2$.
    An analogous argument (for blue degrees) holds in the case when $i$ is odd.
    
    Note that, in such a coloring of $K_n$, an edge $v_1v_j$ is red if $j$ is even and it is blue if $j$ is odd.
    In $T_n$ this translates to a coloring where an arc between $v_1$ and $v_j$ is red if and only if $j$ is even.
    Hence, all red arcs incident to $v_1$ in such a coloring are outgoing from $v_1$.
    Thus, the red indegree of $v_1$ is zero, while the red indegree of $v_2$ (due to the arc $v_1v_2$) is positive. 
    Therefore, in the coloring of $T_n$, vertices $v_1$ and $v_2$ are distinguished by their red outdegree-indegree pairs.
    For other pairs of vertices, the weak local irregularity condition holds since a local irregularity condition holds for these vertices in $K_n$.
    This completes the proof.
\end{proof}

In~\cite{Borowiecki Grytczuk Pilśniak} the following theorem was proved; we use it to prove that for digraphs whose skeletons are complete graphs, three colors are enough to obtain a locally irregular coloring. 

\begin{tw}[Corollary 1 in~\cite{Borowiecki Grytczuk Pilśniak}]\label{Borowiecki1}
    Every graph $G$ can be oriented so that the indegrees of every two adjacent vertices are different.
\end{tw}

\begin{lem}\label{2}
    Let $G$ be a skeleton of a digraph $D$.
    If ${\rm lir}(D')\leq k$ for every orientation $D'$ of $G$ then ${\rm lir}(D)\leq k+1$.
\end{lem}
\begin{proof}
    Let $D''$ be the subdigraph of $D$ induced on arcs going in both directions 
    (i.e., if $xy,yx \in E(D)$ then $x, y \in V(D'')$ and $xy,yx \in E(D'')$). 
    From Theorem~\ref{Borowiecki1} we have a decomposition of $D''$ into a locally irregular digraph $D_1''$ and a remaining digraph $D_2''$. 
    A digraph $D'=(D \setminus D'')\cup D_2''$ is therefore an orientation of $G$. 
    Hence, from the assumption, we have ${\rm lir}(D') \leq k$. Using a new color on the arcs of $D''_1$, we get a locally irregular coloring of $D$ with at most $k+1$ colors.
\end{proof}

\begin{tw}
    If the skeleton of a digraph $D$ is a complete graph then ${\rm lir}(D)\leq 3$.  
\end{tw}
\begin{proof}
    Note that every orientation $D'$ of the skeleton of $D$ is a tournament, hence, Theorem~\ref{tournament} yields ${\rm lir}(D')\leq 2$.
    The result then follows directly from Lemma~\ref{2}.
\end{proof}

For every simple graph $G$ with large minimum degree, namely $\delta(G) \geq 10^{10}$, Przybyło proved that ${\rm lir}(G)\leq 3$. 
As an easy corollary (using Lemma~\ref{2}) of this result, we obtain:
\begin{cor}
    Let $D$ be a digraph such that $d^+(v) + d^-(v) \geq 2\cdot 10^{10}$ for each vertex $v \in V(D)$.
    Then ${\rm lir}(D)\leq 4$.
\end{cor}

Next we present a constant upper bound on locally irregular chromatic index of any orientation of a regular graph. In~\cite{Baudon Bensmail Przybylo Wozniak2}, to prove that ${\rm lir}^{(+,+)}(D) \leq 5$ for each digraph $D$, authors showed that the digraph $D$ can be always decomposed into two acyclic digraphs, one of which is \textit{outdegree-decreasing}, i.e., its vertices can be labeled $v_1, \dots, v_n$ such that:
\begin{itemize}
    \item[1.] $i \leq j$ whenever $v_iv_j$ is an arc, 
    \item[2.] $d^+(v_i) \geq d^+(v_j)$ for each $i < j$.
\end{itemize}
Then they proved that ${\rm lir}^{(+,+)}(D) \leq 2$ whenever $D$ is an outdegree-decreasing acyclic digraph, from which we have:
\begin{lem}[Corollary of Lemma 3.6 from~\cite{Baudon Bensmail Przybylo Wozniak2}]\label{lemma_degree_decreasing}
    If $D$ is an outdegree-decreasing acyclic digraph then ${\rm lir}(D) \leq 2$.
\end{lem}
Note that we can define a \textit{indegree-increasing digraph} as a digraph whose vertices can be labeled $v_1, \dots, v_n$ such that:
\begin{itemize}
    \item[1.] $i \leq j$ whenever $v_jv_i$ is an arc,
    \item[2.] $d^-(v_i) \leq d^-(v_j)$ for each $i \leq j$. 
\end{itemize}
Observe that changing the orientations of all arcs in an indegree-increasing acyclic digraph yields an outdegree-decreasing acyclic digraph. 
Hence, we get the immediate corollary of Lemma~\ref{lemma_degree_decreasing}:
\begin{cor}\label{lemma_degree_increasing}
    If $D$ is an indegree-increasing acyclic digraph then ${\rm lir}(D) \leq 2$.
\end{cor}

Each digraph $D$ which is an orientation of regular graph can be decomposed into two acyclic digraphs, one of which is outdegree-decreasing and the other one is indegree-increasing. To prove this, we use a greedy method described in~\cite{Baudon Bensmail Przybylo Wozniak2}:
Start with a vertex $v_1$ with the largest outdegree, and in $i$-th step, for $i \geq 2$, take a vertex from $D \setminus \{v_1, \dots, v_{i-1}\}$ with the largest outdegree and label it $v_i$. 
A digraph $D_{\text{dec}}$ induced on all arcs $v_iv_j$ of $D$ such that $i \leq j$ is acyclic and outdegree-decreasing. 
Let us consider a subdigraph $D_{\text{inc}}$ of $D$ induced on the remaining arcs. 
We claim that $D_{\text{inc}}$ is acyclic and indegree-increasing.
To show that $D_{\text{inc}}$ is acyclic, suppose that there is a cycle, and let $v_i$ be the vertex which was labeled first among all vertices on this cycle. 
All arcs outgoing from $v_i$ to vertices $v_j$ where $j > i$ are present in $D_{\text{dec}}$. 
Hence, non of such arcs are present in $D_{\text{inc}}$. 
But if $v_i$ is on a cycle in $D_{\text{inc}}$, there should be at least one arc outgoing from $v_i$ in $D_{\text{inc}}$, a contradiction.
The fact that $D_{\text{inc}}$ is indegree-increasing follows directly from the assumption that $D$ is an orientation of a regular graph and the construction of $D_{\text{dec}}$ and $D_{\text{inc}}$.

Hence, an orientation of a regular graph can be decomposed into an outdegree-decreasing and an indegree-increasing acyclic subdigraphs, and from Lemma~\ref{lemma_degree_decreasing} and Corollary~\ref{lemma_degree_increasing} we get:
\begin{tw}
For every digraph $D$ which is an orientation of regular graph we have ${\rm lir}(D)\leq 4$.  
\end{tw}

A digraph $D$ is called \textit{symmetric} if, for each pair of vertices $u,v \in V(D)$, $uv$ is an arc of $D$ if and only if $vu$ is an arc of $D$. 
We use the following theorem, proved by Borowiecki, Grytczuk, and Pilśniak in~\cite{Borowiecki Grytczuk Pilśniak} to prove that Conjecture~\ref{main} is true for symmetric digraphs.

\begin{tw}[Corollary 2 in~\cite{Borowiecki Grytczuk Pilśniak}]\label{Borowiecki2}
    Every undirected graph $G$ has an orientation in which every two adjacent vertices have different balanced degrees.
\end{tw}

\begin{tw}\label{lir_symmetric}
    For every symmetric digraph $D$ we have ${\rm lir}(D)\leq 2$.  
\end{tw}
\begin{proof}
    Let $G$ be a skeleton of $D$.
    From Theorem~\ref{Borowiecki2} we have that there is an orientation $D'$ of $G$ in which adjacent vertices have different balanced degrees.
    Note that $D'$ is a subdigraph of $D$, and if $d^+(u) - d^-(u) \neq d^+(v) - d^-(v)$ then $d^+(u) \neq d^+(v)$ or $d^-(u) \neq d^-(v)$. 
    Hence, $D'$ is a locally irregular subdigraph of $D$.

    Let $D'' = D \setminus D''$. 
    Observe that $D''$ is also an orientation of $G$, and the outdegree of $v$ in $D''$ equals its indegree in $D'$ and the indegree of $v$ in $D''$ equals its outdegree in $D'$. 
    Hence, in $D''$ adjacent vertices are distinguished using the outdegree-indegree pair.
    Thus, a coloring of $D$ in which all arcs of $D'$ are red and all arc of $D''$ are blue is locally irregular.
\end{proof}

\section{Decomposing digraphs into strongly locally irregular subdigraphs}\label{section_strong}

In this section we investigate the strong variant of locally irregular decompositions and colorings. 
Recall a digraph is strong locally irregular if for every its arc $xy$ balanced degrees of $x$ and $y$ are different, i.e., $\sigma(x) = d^+(x) - d^-(x) \neq d^+(y) - d^-(y) = \sigma(y)$.
A definition of strong locally irregular coloring (or a decomposition) is analogous to the ones defined for other local irregularity concepts: an arc coloring  of a digraph is strong locally irregular if each color class induces a strong locally irregular digraph.
The minimum number of colors in a strong locally irregular coloring of a digraph $D$, a strong locally irregular chromatic index of $D$, is denoted by $\mathrm{slir}(D)$.

Similarly to weak and $(+,+)$-local irregularities, a digraph consisting of a single arc is strong locally irregular.
Hence $\mathrm{slir}(D)$ is well defined for every digraph $D$, and it is trivially upper bounded by $|E(D)|$. 
However, in contrast to weak and $(+,+)$-locally irregular colorings, there is not a straight-forward connection between a locally irregular chromatic index of a simple graph and a strong locally irregular chromatic index of its orientation.
To see this, consider a star $K_{1,3}$: while $K_{1,3}$ is locally irregular itself, an orientation of it, where a single edge $e$ is oriented towards the central vertex $v$ and the remaining two edges are oriented from $v$, contains an arc (the orientation of $e$) whose both ends have their balanced degrees equal to one.
Hence, even proving Conjecture~\ref{graph3} would not imply that (weakly connected) oriented graphs (with an exception of orientations of particular cacti) have strong locally irregular chromatic index bounded by 3. Despite this, we still believe the following conjecture is true:

\begin{cnj}
\label{main1}
Every digraph $D$ satisfies ${\rm slir}(D)\leq 2$.
\end{cnj}

In the next two subsections we provide results that supports Conjecture~\ref{main1}.

\subsection{Simple classes of digraphs and a general upper bound}

First, we show that Conjecture~\ref{main1} is true for symmetric digraphs, and digraphs with bipartite and tripartite skeletons. Proofs of these results are similar to those proved in the case of weak local irregularity, see Theorem~\ref{dwudzielne}, Theorem~\ref{trojdzielne}, and Theorem~\ref{lir_symmetric}.

\begin{tw}
For every symmetric digraph $D$ we have ${\rm slir}(D)\leq 2$.  
\end{tw}
\begin{proof}
Let $G$ be a skeleton of $D$. 
From Theorem~\ref{Borowiecki2} we have an orientation $D'$ of $G$ which is strong locally irregular.
$D'$ is a subdigraph of $D$, and by removing all arc of $D'$ from $D$ we are left with a subdigraph of $D$, where $uv \in E(D'')$ if and only if $vu \in E(D')$.
Hence, the balanced degree of a vertex $v$ in $D''$ is the opposite number to the balanced degree of $v$ in $D'$.
Thus, $D'$ and $D''$ form a strong locally irregular decomposition of $D$.
\end{proof}

\begin{lem}\label{lemma_sliec_bipartite}
    If $D$ is a digraph with a bipartite skeleton then $\operatorname{slir}(D) \leq 2$.
\end{lem}
\begin{proof}
    Let $X,Y$ be a bipartition of vertices of a skeleton of $D$.
    Color each arc from $X$ to $Y$ red, and each arc from $Y$ to $X$ blue.
    In the red subdigraph, except isolated vertices, vertices of $X$ have positive balanced degrees, while the vertices of $Y$ have negative balanced degrees.
    Similarly, in the blue subdigraph, non-isolated vertices of $X$ have negative balanced degrees and non-isolated vertices of $Y$ have positive balanced degrees.
    Thus, the coloring is strong locally irregular.
\end{proof}

\begin{lem}
    If a skeleton of a digraph $D$ is a tripartite graph then ${\rm slir}(D)\leq 3$.
\end{lem}
\begin{proof}
    Let $X,Y,Z$ be a tripartition of the skeleton of $D$. 
    Color all arcs outgoing from vertices of $X$ blue, all arcs outgoing from vertices of $Y$ red and all arcs outgoing from vertices of $Z$ green. 
    After this, every vertex from $X$ has a nonnegative blue balanced degree, and nonpositive red and green balanced degrees.
    Every vertex from $Y$ has a nonnegative red balanced degree, and nonpositive blue and green balanced degrees.
    Every vertex from $Z$ has a nonnegative green balanced degree, and nonpositive blue and red balanced degrees. 
    Moreover, only vertices not incident to blue, red or green arcs have their blue, red or green balanced degrees equal to zero. 
    Thus, the coloring is strong locally irregular.
\end{proof}

Using the previous lemma, and Lemma~\ref{3} which states that $k$-chromatic simple graph can be written as an edge-disjoint union of $\lceil \log_p k \rceil$ $p$-partite graphs, we immediately obtain the following upper bound on strong locally irregular index.

\begin{tw}
For every digraph $D$ we have ${\rm slir}(D)\leq 3\lceil log_3k\rceil$, where $k$ is the chromatic number of the skeleton of $D$.  
\end{tw}

One of the classes of digraphs for which Conjecture~\ref{main1} holds is the class of \textit{Eulerian digraphs}, i.e., digraphs in which an Eulerian circuit (a directed trail that traverses all arcs) exists. 
For Eulerian digraphs an equivalent definition exists: Eulerian digraph is a connected digraph (in the class of Eulerian digraphs weak connectivity is equivalent to the strong one) in which $d^+(v) = d^-(v)$ for each vertex $v$.
To prove that ${\rm slir}(D) \leq 2$ for each Eulerian digraph $D$, we first modify the result of Borowiecki, Grytczuk, and Pilśniak~\cite{Borowiecki Grytczuk Pilśniak}.

\begin{lem}[Modification of Theorem 3 in~\cite{Borowiecki Grytczuk Pilśniak}]\label{lemma_modification_Borowiecki}
    Let $k \geq 2$ be an integer.
    For every digraph $D$ there is an assignment $c \colon E(D) \to \{1,k\}$ such that 
    $\varphi \colon V(G) \to \mathbb{Z}$ defined by $\varphi(v) = \sum\limits_{vx \in E(D)}c(vx) - \sum\limits_{yv \in E(D)} c(yv)$ distinguishes neighbors $($i.e., $\varphi(u) \neq \varphi(v)$ whenever $uv \in E(D)$$)$.
\end{lem}
\begin{proof}
    We start with an initial chip configuration: each arc of $D$ is assigned one red and $k$ blue chips.
    Hence, in the beginning, there are no chips on vertices of $D$.
    We provide a simple procedure, during which we add some red chips and remove some blue chips on arcs incident with a vertex, and then shift them from or to this vertex.
    This shifting follows two simple rules:
    \begin{itemize}
        \item[(R1)]  blue chips are always shifted in a direction of the arc they are on, while red chips moves against the direction,
        \item[(R2)] right before shifting the chips from an arc to the incident vertices, we either add red chips or remove blue chips so the numbers of red and blue chips on the arc are equal.
    \end{itemize}
    Moreover, let $q^-_B(v)$ be the number of blue chips on $v$ and the arcs ingoing to $v$, and let $q^+_R(v)$ be the number of red chips on $v$ and on the arcs outgoing from $v$.
    The potential $q(v)$ of $v$ is then defined by $q(v) = q^-_B(v) - q^+_R(v)$.
    Let $D_0 = D$.

    In the $i$-th step of the procedure, for $i$ from 1 to $n$, a single vertex $v_i$ is taken from $D_{i-1}$ such that $q(v_i)$ is the maximum possible.
    We add $k-1$ red chips to every arc ingoing to $v_i$ and remove $k-1$ blue chips to every arc outgoing from $v_i$. 
    After that, chips on arcs incident to $v_i$ are shifted, following the rule (R1), and we set $D_i = D_{i-1} - v_i$. 

    Let $v_iv_j$ be an arc of $D$ for some $i < j$.
    In $D_{i-1}$ we have $q(v_i) \geq q(v_j)$.
    In the $i$-th step, we remove $k-1$ blue chips from $v_iv_j$ and shift the remaining one red chip and one blue chip on this arc according to (R1); after this, the potential of $v_i$ remains unchanged, while the potential of $v_j$ is lowered (instead of adding $k-1$ for blue chips on $v_iv_j$ in $D_{i-1}$, we add only 1).
    Similarly, if $v_jv_i$ is an arc of $D$ for some $i < j$ then right before the $i$-th step we have $q(v_i) \geq q(v_j)$. 
    Then, in the $i$-th step, we add $k-1$ red chips to $v_jv_i$ and apply (R1). 
    This, once again, leaves the potential of $v_i$ unchanged and lowers the potential of $v_j$.
    
    Hence, during the run of the procedure, potentials of individual vertices do not grow and, after the $i$-th step, potentials of vertices from $V(D) \setminus \{v_1, \dots, v_i\}$ adjacent to $v_i$ are strictly smaller then the potential of $v_i$ right before $v_i$ is removed.
    Hence, the neighboring vertices of $D$ are distinguished by a function $q'$ which assigns to every $v_i$ the value $q(v_i)$ obtained right before the end of the $i$-th step.

    To obtain a coloring from $q'$, we simply shift all red chips from vertices back to incident outgoing arcs: for an arc $v_iv_j$, if $i < j$ then $v_i$ sends a red chip to $v_iv_j$, and if $i >j$ then $v_j$ sends $k$ red chips to $v_iv_j$.
\end{proof}

\begin{tw}
    If $D$ is an Eulerian digraph then ${\rm slir}(D) \leq 2$.
\end{tw}
\begin{proof}
    Let $k$ be an integer such that $k > \max\{\Delta^+(D), \Delta^-(D)\}$, where $\Delta^+(D)$ and $\Delta^-(D)$ stands for the maximum outdegree and the maximum indegree of a vertex in $D$, respectively.
    For such $k$, consider an assignment $c \colon E(D) \to \{1,k\}$ and the corresponding neighbor-distinguishing mapping $\varphi \colon V(D) \to \mathbb{Z}$ from Lemma~\ref{lemma_modification_Borowiecki}.

    For each vertex $v$ of $D$ we define:
    \begin{align*}
        r^+(v) = |\{vx \in E(D) \colon \varphi(vx) = 1\}|,\\
        r^-(v) = |\{yv \in E(D) \colon \varphi(yv) = 1\}|,\\
        b^+(v) = |\{vx \in E(D) \colon \varphi(vx) = 2\}|,\\
        b^-(v) = |\{yv \in E(D) \colon \varphi(yv) = 2\}|.
    \end{align*}
    Moreover, let $r(v) = r^+(v) - r^-(v)$ and $b(v) = b^+(v) - b^-(v)$.
    Then $\varphi(v) = r(v) + k \cdot b(v)$ for each $v$.
    
    Since $k > \max\{\Delta^+(D), \Delta^-(D)\}$ we have that $k > |r(v)|$ and $k > |b(v)|$.
    Hence, if $\varphi(u) \neq \varphi(v)$, then $r(u) \neq r(v)$ or $b(u) \neq b(v)$.
    We show that since $D$ is Eulerian, $r(u) \neq r(v)$ implies $b(u) \neq b(v)$, and vice versa.

    Suppose $r(u) \neq r(v)$. 
    For every vertex $x$ of $D$ we have $b(x) = b^+(x) - b^-(x) = (d^+(x) - r^+(x)) - (d^-(x) - r^-(x)) = (d^+(x) - d^-(x)) + (r^-(x) - r^+(x))$.
    Since $D$ is Eulerian, $d^+(x) - d^-(x) = 0$ and we get $b(x) = r^-(x) - r^+(x) = - r(x)$. 
    In particular, for vertices $u$ and $v$ we get $b(u) = -r(u)$ and $b(v) = -r(v)$ which, using the fact $r(u) \neq r(v)$, implies $b(u) \neq b(v)$.
    The implication $b(u) \neq b(v) \implies r(u) \neq r(v)$ is analogous.

    Hence, $r(u) \neq r(v)$ and $b(u) \neq b(v)$ for each arc $uv$ in $D$.
    We now define a coloring of arcs of $D$ in the following way: each arc $uv$ is colored red whenever $\varphi(uv) = 1$, and it is colored blue whenever $\varphi(uv) = 2$.
    For such a coloring, the balanced red degree and the balanced blue degree of a vertex $v$ equal $r(v)$ and $b(v)$, respectively.
    From what we proved, both balanced color degrees are different for every two adjacent vertices, yielding that such a coloring is a strong locally irregular coloring.
\end{proof}

\subsection{Orientations of cacti}

When studying a variation of locally irregular colorings and decompositions, a somewhat interesting class of graphs (or digraphs in our case) is the class of cacti (the corresponding class of digraph).
This is due to the special position of cacti among all the simple graphs since every graph that does not have a locally irregular decomposition is a cactus (see for example Observation 3.3 in~\cite{Baudon Bensmail Przybylo Wozniak}, Corollary 26 in~\cite{Sedlar Skrekovski 2}, Theorem 4 in~\cite{Sedlar Skrekovski}).

We start with some easy observations on the orientations of stars.
Suppose that $D$ is an orientation of $K_{1,k}$.
Denote by $v$ the central vertex of $K_{1,k}$, and by $a$ and $b$ the outdegree and the indegree of $v$ in $D$, respectively.
If we color $i$ arcs that are outgoing from $v$ blue, $j$ arcs that are ingoing to $v$ blue, and the all the remaining arcs red, we obtain a coloring in which the blue balanced degree of $v$ equals $i-j$.
For different values of $i \in \{0, \dots, a\}$ and $j \in \{0,\dots, b\}$ we are able to obtain a coloring where the blue balanced degree is any integer from $\{-b, \dots, a\}$.
Note that if the blue balanced degree of $v$ is $x$ then the red balanced degree of $v$ is $a-b - x$.
Hence, the above-listed colorings yield $k+1$ pairs of obtainable blue and red degrees of $v$, where any two pairs differ in their balanced blue and red degrees. 
We get the following:
\begin{obs}\label{obs_oriented_star_1}
    If $D$ is an orientation of $K_{1,k}$ then there are $k+1$ arc colorings of $D$ with two colors that produce element-wise distinct balanced color degree pairs in the central vertex. 
\end{obs}

For the second easy observation, consider an orientation $D$ of a star $K_{1,k}$.
Suppose that $D$ is not a locally irregular digraph.
This means that there is an arc $uv$ such that $\sigma(u) = \sigma(v)$.
However, since $D$ is the orientation of $K_{1,k}$, we get that one of the two vertices $u$ and $v$ is the central vertex, and the other one is pendant in $K_{1,k}$. 
Thus, $\sigma(u) = \sigma(v) = 1$ (if $u$ is pendant and $v$ is central) or $\sigma(u) = \sigma(v) = -1$ (if $u$ is central and $v$ is pendant).
In both cases we get that the outdegree and the indegree of the central vertex of $K_{1,k}$ differ by 1 in $D$.
Considering the fact that a single oriented edge is locally irregular, we get the following:
\begin{obs}\label{obs_oriented_star_2}
    Let $v$ be a central vertex of $K_{1,k}$, and let $D$ be an orientation of $K_{1,k}$.
    If $D$ is not locally irregular then $d^+(v)$ and $d^-(v)$ are positive and $d^+(v) - d^-(v) \in \{-1,1\}$.
\end{obs}

\begin{lem}\label{only_stars_are_interesting}
    Let $v$ be a pendant vertex of a simple graph $G$.
    Let $D$ be an orientation of $G$, and let $D'$ be an oriented graph obtained from $D$ by attaching some number of pendant arcs outgoing from $v$ and some number of pendant arcs ingoing to $v$. 
    If $\mathrm{slir}(D) \leq 2$ then $\mathrm{slir}(D') \leq 2$.
\end{lem}
\begin{proof}
    Denote by $w$ the only neighbor of $v$ in $D$.
    Let $S$ be an oriented star that consists of all arcs added to $D$ to create $D'$.
    Consider a red-blue strongly locally irregular coloring $\varphi$ of $D$ in which the arc between $v$ and $w$ is blue.
    We show how that there is a coloring of $S$ which together with $\varphi$ results in a strongly locally irregular coloring of $D'$.
    
    Denote by $a$ and $b$ the outdegree and the indegree of $v$ in $S$, respectively.
    Both $a$ and $b$ are positive, and their difference is $1$ or $-1$, see Observation~\ref{obs_oriented_star_1}.
    Suppose that $a - b = 1$ (the second case is similar). 
    We have $a \geq 2$.

    If $S$ is strongly locally irregular, color each its arc red. It is easy to see that the obtained coloring of $D'$ is strongly locally irregular.

    Suppose therefore that $S$ is not strongly locally irregular.
    We consider two different colorings $\varphi_1$ and $\varphi_2$ of $S$: in $\varphi_1$ we color a single arc $vu$ (an arc outgoing from $v$) blue, and the remaining arcs of $S$ red, and in $\varphi_2$ we color all arcs outgoing from $v$ blue, and all arc ingoing to $v$ red.
    In both cases, the red subdigraph of $S$ is locally irregular.
    Observe also that, in both cases, there is no conflict between $v$ and its outneighbor (take $u$ as a representative) in $V(S)$.
    When combining $\varphi$ with $\varphi_1$, the blue balanced degree of $u$ is $-1$, and the blue balanced degree of $v$ is $0$ or $2$ (as there are two blue arcs incident to $v$ and at least one of them is outgoing from $v$).
    When combining $\varphi$ with $\varphi_2$, the blue balanced degree of $u$ is $-1$, and the blue balanced degree of $v$ is at least $1$ (there are at least $a \geq 2$ blue arcs outgoing from $v$ and at most one blue arc ingoing to $v$).

    Hence, if the strongly locally irregular condition is violated in the coloring obtained by combining $\varphi$ with $\varphi_1$ or $\varphi_2$, there is a conflict between $v$ and $w$.
    However, $\varphi_1$ and $\varphi_2$ change the blue balanced degree of $v$ in two different ways ($\varphi_1$ raises it by 1, and $\varphi_2$ raises it by at least 2).
    Thus, at least one of the considered colorings of $S$ does not create a conflict between $v$ and $w$ in $D'$; we can extend $\varphi$ to a strongly locally irregular coloring of $D'$ without introducing new colors.
\end{proof}

In a cactus $G$, after removing all edges that lie on cycles, we are left with a graph $H$ (subgraph of $G$ induced on bridges of $G$), whose components are trees. 
Each of these trees has at least one vertex which lies on a cycle in $G$; if such a tree $T$ has exactly one vertex which lies on a cycle in $G$, and it has at least two vertices in total, we say that $T$ is a \textit{pendant tree} of $G$.
Alternatively, pendant trees can be seen as trees attached to a single vertex of a cycle in $G$.

In the proofs of the following results on strong locally irregular colorings of oriented cacti, the main strategy is to suppose to the contrary that the claims are not true, and to consider the minimal counterexample, i.e., the smallest cactus $G$ which has an orientation $D$ with $\mathrm{slir}(D) \geq 3$.
Consider a pendant tree $T$ of  $G$ attached to a vertex $v$: if $T$ is not a star, then we may remove all its edges except those which are incident $v$. From the assumption that $G$ is the smallest counterexample we get that any orientation of the new graph has a strong locally irregular coloring with at most 2 colors.
Such a coloring is then extended to a coloring of $D$ using Lemma~\ref{only_stars_are_interesting}, as we can, level by level, attach a star to a pendant vertex, and always find a suitable coloring. 
Hence, we get the following observation:
\begin{obs}\label{obs_min_bad_cactus}
    If $G$ is the smallest cactus  which has an orientation $D$ with $\mathrm{slir}(D) \geq 3$ then every pendant tree of $G$ is a star.
\end{obs}

Now we show that Conjecture~\ref{main1} is true for orientations of unicyclic graphs.

\begin{tw}\label{thm_slir_unicyclic}
    Let $D$ be an orientation of a unicyclic graph $G$. Then $\mathrm{slir}(D) \leq 2$.    
\end{tw}
\begin{proof}
    Suppose to the contrary that the theorem is false.
    Let $G$ be the smallest unicyclic graph which has an orientation $D$ with $\mathrm{slir}(D) \geq 3$.
    Hence, $G$ is connected and from the Observation~\ref{obs_min_bad_cactus} we have that every pendant tree of $G$ is a star.
    Denote by $C$ the only cycle in $G$.
    The length of $C$ is odd, otherwise $G$ is bipartite and Lemma~\ref{lemma_sliec_bipartite} yields a contradiction.
    
    Since $C$ is odd, there are three consecutive vertices $u,v,w$ on $C$ such that $uv,vw \in E(D)$.
    Denote by $T$ the pendant tree attached to $v$ (if the degree of $v$ in $G$ is 2, let $T$ be a graph with a single vertex $v$).
    Consider now a digraph $D'$ obtained from $D$ by removing all vertices of $T$ (including $v$), adding two new vertices $v_1$ and $v_2$, and adding arcs $uv_1$ and $v_2w$.
    Skeleton of $D'$ is a tree, hence, we can color its arcs in the same way that is described in the proof of Lemma~\ref{lemma_sliec_bipartite}: Let $X,Y$ be the bipartition of vertices of the skeleton of $D'$. Color each arc from $X$ to $Y$ blue, and each arc from $Y$ to $X$ red. Such a coloring of $D'$ is strong locally irregular.
    
    Suppose, without loss of generality, that $u \in X$.
    Then $w \in Y$ and thus, both arcs  $uv_1$ and $v_2w$ are blue.
    Moreover, blue balanced degrees of $u$ and $w$ are positive and negative, respectively.
    Such a coloring of $D$ naturally induces a strong locally irregular coloring of $D''$ -- a digraph obtained from $D$ by removing all vertices of $T$ except $v$.
    In the coloring of $D''$, arcs $uv_1$ and $v_2w$ are blue, and blue balanced degrees of vertices $u$, $v$, and $w$ are positive, equal to zero, and negative, respectively.
    
    If the degree of $v$ in $G$ is 2, we have a strong locally irregular coloring of $G$ with 2 colors, which contradicts $\mathrm{slir}(D) \geq 3$.
    Hence, suppose that the degree of $v$ in $G$ is at least 3, i.e., $T$ has at least two vertices.
    We proceed to the coloring of arcs that correspond to edges of the star $T$; we call these arcs pendant.
    
    Recall that both arcs incident to $v$ which are not pendant are blue.
    Hence, if the orientation $S$ of $T$ is strong locally irregular, we may color all its arcs red, and obtain a strong locally irregular coloring of $D$ with two colors.
    Therefore, suppose that $S$ is not strong locally irregular.
    From Observation~\ref{obs_oriented_star_2} we have that there is at least one arc outgoing from $v$ and at least one arc ingoing to $v$ in $S$, and $d^+(v) - d^-(v) \in \{-1,1\}$ in $S$.
    
    Suppose that in $S$ we have $d^+(v) - d^-(v) = 1$ (the other case is similar).
    Consider two different colorings $\varphi_1$ and $\varphi_2$ of $S$. 
    In $\varphi_1$ a single arc of $S$ that is outgoing from $v$ is colored blue and the remaining arcs of $S$ are red.
    In $\varphi_2$ all arcs outgoing from $v$ are blue, and all arcs ingoing to $v$ are red.
    
    Observe that in both colorings $\varphi_1$ and $\varphi_2$ of $S$, the red subdigraph is locally irregular.
    Moreover, after we combine these colorings of $S$ with the coloring of $D''$ (in which $v$ has blue balanced degree equal to zero), we get a coloring in which the blue balanced degree of $v$ is positive.
    Hence, if using $\varphi_1$ or $\varphi_2$  on arcs of $S$ in $D$ results in a violation of the strong local irregularity condition, the conflict is between blue balanced degrees of $v$ and $u$.
    However, the blue balanced degree of $u$ is the same as in the coloring of $D''$, and colorings $\varphi_1$ and $\varphi_2$ change the blue balanced degree of $v$ from 0 in the coloring of $D''$ to two different positive integers in the coloring of $D$.
    Thus, at least one of the colorings $\varphi_1$ and $\varphi_2$ can be applied to arcs of $S$ to extend the strong locally irregular coloring of $D''$ to a strong locally irregular coloring of $D$ which uses two colors, a contradiction.
\end{proof}

The previous theorem implies that, if a counterexample to Conjecture~\ref{main1} among orientations of cacti exists, then the cactus contains at least two cycles.
We further restrict how such a minimum counterexample could look like. 
A cycle of a cactus is called pendant if the removal of its edges yields a graph in which at most one component has cycles.
Note that every cactus contains a pendant cycle unless it is a tree (this can be easily seen: repeatedly remove the pendant edges of a cactus, and at the point where there are no pendant edges, pendant cycles correspond to leaves of the block-cut tree).

\begin{lem}\label{lemma_counterexample_cactus}
    Let $G$ be the minimum cactus $($concerning the number of vertices and edges$)$ such that there is an orientation $D$ of $G$ with $\operatorname{slir}(D) \geq 3$. 
    Then:
    \begin{itemize}
        \item[(i)] Every pendant cycle of $G$ has length $3$.
        \item[(ii)] If $G$ contains a pendant cycle $C$ of length $3$ such that $G - V(C)$ is a graph with at most one component containing cycles, then there is a vertex of $C$ that lies on multiple cycles in $G$.
    \end{itemize}
\end{lem}

\begin{proof}
    First, observe that $G$ has at least two cycles (see Theorem~\ref{thm_slir_unicyclic}), and all pendant trees of $G$ are stars (Observation~\ref{obs_min_bad_cactus}). 

    Let $C$ be a pendant cycle of the maximum length in $G$.
    Let $x$ be a vertex of $C$ such that the remaining vertices of $C$ lie in acyclic components of $G - E(C)$.
    Let $v_1 \in V(C) \setminus \{x\}$ has the maximum degree.
    Suppose that $d(v_1) = k+2$, i.e., $v_1$ is incident to $k$ pendant edges in $G$.
    
    By $S$ we denote the star digraph induced on $k$ pendant arcs incident to $v_1$ in $D$.
    Let $D_1$ be a digraph obtained from $D$ by removing arcs of $S$ and all the newly obtained isolated vertices (pendant neighbors of $v_1$). 
    From the assumption, there is a strong locally irregular coloring $\varphi_1$ of $D_1$ with at most 2 colors.
    Let $b_1$ and $r_1$ be balanced blue and red degrees of $v_1$ in $\varphi_1$.
    In what follows we discuss how the arcs of $S$ could be colored in such a way that this coloring would extend $\varphi_1$ into a strong locally irregular coloring with at most two colors of $D$, considering that $k$ is large enough. 
    This yields an upper bound on $k$. 

    We claim that $k \leq 5$. 
    Suppose this is not the case, i.e., at least 6 pendant edges are attached to $v_1$ in $G$.
    Note that assigning colors to $k$ pendant arcs incident to $v_1$ changes the color degrees of $v_1$ to $b_1+b'$ and $r_1 + r'$, where $b'$ and $r'$ are balanced blue and red degrees of $v_1$ in the coloring of $S$.
    For each pair of the final balanced blue and red degrees $b_1 + b'$ and $r_1 + r'$ in $v_1$, there are six possible conflicts with the strong locally irregular condition: two on the arcs of $C$, one for blue pendant arcs outgoing from $v_1$ (if $b_1+b'=-1$), one for red pendant arcs outgoing from $v_1$ (if $r_1 + r' = -1$), one for blue pendant arcs ingoing to $v_1$ (if $b_1 + b' = 1$), and one for red pendant arcs ingoing to $v_1$ (if $r_1 + r' = 1)$.
    
    If $k \geq 6$, Observation~\ref{obs_oriented_star_1} yields that there are at least 7 different colorings of $S$, each changing both blue and red balanced degrees of $v_1$ differently to others.
    Hence, at least one of them changes the blue and red balanced degrees of $v_1$ in such a way, that none of the six possible conflicts occur.
    Such a coloring, together with $\varphi_1$, results in a strong locally irregular coloring of $D$ with 2 colors, contradicting $\mathrm{slir}(D) \geq 3$. 
    Hence, $k \leq 5$.

    It is possible to further lower the bound on $d(v_1)$ by considering the possibilities of how the edges incident to $v_1$ on $C$ are oriented in $D$ and colored in $\varphi_1$, what the red and blue balanced degrees of neighbors $v_0$ and $v_2$ of $v_1$ on $C$ are, and what the orientations of arcs on $S$ are. 
    If for example the arc between $v_0$ and $v_1$ is blue, the value of the red balanced degree of $v_0$ is unimportant. 
    Moreover, if the blue balanced degree of $v_0$ is greater than 8 or smaller than $-8$, 
    no coloring of $S$ changes the blue balanced degree of $v_1$ in a way that would violate the strong locally irregular condition on the arc between $v_0$ and $v_1$ (since $d(v_1) \leq 7$).
    This leaves us with a finite number of cases that need to be considered; using a simple Python program we generate these cases, and for each of them we find a suitable coloring of $S$ (see 'S2 Proof of Claim 1' in~\cite{supplementary_material}). 
    In particular, by doing this we prove: 
    \begin{claim}\label{claim1}
        $d(v_1) \leq 4$.    
    \end{claim}

    In the following, suppose that $v_2$ is different from $x$; if $v_1$ is not adjacent to $x$, further assume that $d(v_0) \leq d(v_2)$.
    Since $d(v_1) \leq 4$ and $v_1$ has the largest degree among vertices from $V(C) \setminus \{x\}$, we have $d(v_2) \leq 4$, too.
    For $v_1$ and $v_2$ we have $(d(v_1),d(v_2)) \in \{ (4,4),(4,3),(4,2),(3,3),(3,2)\}$, or all vertices of $C$ except $x$ have degree 2 in $G$.

    Using a Python program, see 'S3 Proof of Claim 2' in~\cite{supplementary_material}, we prove:
    
    \begin{claim}\label{claim2}
        $(d(v_1),d(v_2)) \notin \{ (4,4),(4,3),(4,2)\}$.
    \end{claim}
    \noindent By showing that in these cases of $d(v_1)$ and $d(v_2)$ a strong locally irregular coloring of a digraph $D_2$ obtained from $D$ by removing an arc between $v_1$ and $v_2$, and all pendant neighbors of $v_1$ and $v_2$, can be extended to a locally irregular coloring with two colors of $D$.
    Once again, the number of considered cases is finite due to the degrees of $v_1$ and $v_2$ in $G$.
    
    For $(d(v_1),d(v_2)) = (3,3)$ we show using a Python program (see 'S4 Proof of Claim 3' in~\cite{supplementary_material}) that using the same approach as to prove Claim~\ref{claim2} results in four cases (shown in Figure~\ref{fig:four_unsolved_cases_by_computer}) for which a strong locally irregular coloring of $D_3$ (a digraph obtained from $D$ by removing the arc between $v_1$ and $v_2$, and pendant neighbors of $v_1$ and $v_2$) cannot be extended to a strong locally irregular coloring with two colors of $D$.
    However, in these cases, the considered color degrees of $v_0$ and $v_3$ (a neighbor of $v_2$ on $C$ different from $v_1$), and colors and orientations of arcs incident to them yields $d(v_0) \geq 4$ and $d(v_3) \geq 4$: to see this, consider the case when a blue arc is going from $v_1$ to $v_0$ and the blue balanced degree of $v_0$ is 2 (the remaining cases are analogous) -- at least three arcs are necessarily outgoing from $v_0$ in $D_3$.
    If the length of $C$ is at least 4 then $v_0 \neq x$ or $v_3 \neq x$, which contradicts the choice of $v_1$ as a vertex of maximum degree among all vertices from $V(C) \setminus \{x\}$.
    We get:
    \begin{claim}
        If $C$ is of length at least $4$ then $(d(v_1),d(v_2)) \neq (3,3)$. 
    \end{claim}

    \begin{figure}[h]
        \centering
        \begin{tikzpicture}
	\begin{pgfonlayer}{nodelayer}
		\node [style={black_dot}] (0) at (-4, 2) {$v_0$};
		\node [style={black_dot}] (1) at (-3, 1) {$v_1$};
		\node [style={black_dot}] (2) at (-1.5, 1) {$v_2$};
		\node [style={black_dot}] (3) at (-0.5, 2) {$v_3$};
		\node [style={black_dot}] (4) at (-3, -0.25) {$v_4$};
		\node [style={black_dot}] (5) at (-1.5, -0.25) {$v_5$};
		\node [style=empty-vertex] (6) at (-4, 2.75) {};
		\node [style=empty-vertex] (7) at (-4.5, 2.5) {};
		\node [style=empty-vertex] (8) at (-3.5, 2.5) {};
		\node [style=empty-vertex] (9) at (-0.5, 2.75) {};
		\node [style=empty-vertex] (10) at (-1, 2.5) {};
		\node [style=empty-vertex] (11) at (0, 2.5) {};
		\node [style=empty-vertex] (12) at (-4.75, 2) {$\color{blue}{2}$};
		\node [style=empty-vertex] (14) at (0.25, 2) {$\color{red}{-2}$};
		\node [style={black_dot}] (15) at (2.5, 2) {$v_0$};
		\node [style={black_dot}] (16) at (3.5, 1) {$v_1$};
		\node [style={black_dot}] (17) at (5, 1) {$v_2$};
		\node [style={black_dot}] (18) at (6, 2) {$v_3$};
		\node [style={black_dot}] (19) at (3.5, -0.25) {$v_4$};
		\node [style={black_dot}] (20) at (5, -0.25) {$v_5$};
		\node [style=empty-vertex] (21) at (2.5, 2.75) {};
		\node [style=empty-vertex] (22) at (2, 2.5) {};
		\node [style=empty-vertex] (23) at (3, 2.5) {};
		\node [style=empty-vertex] (24) at (6, 2.75) {};
		\node [style=empty-vertex] (25) at (5.5, 2.5) {};
		\node [style=empty-vertex] (26) at (6.5, 2.5) {};
		\node [style=empty-vertex] (27) at (1.75, 2) {$\color{red}{2}$};
		\node [style=empty-vertex] (28) at (6.75, 2) {$\color{blue}{-2}$};
		\node [style={black_dot}] (29) at (-4, -2.25) {$v_0$};
		\node [style={black_dot}] (30) at (-3, -3.25) {$v_1$};
		\node [style={black_dot}] (31) at (-1.5, -3.25) {$v_2$};
		\node [style={black_dot}] (32) at (-0.5, -2.25) {$v_3$};
		\node [style={black_dot}] (33) at (-3, -4.5) {$v_4$};
		\node [style={black_dot}] (34) at (-1.5, -4.5) {$v_5$};
		\node [style=empty-vertex] (35) at (-4, -1.5) {};
		\node [style=empty-vertex] (36) at (-4.5, -1.75) {};
		\node [style=empty-vertex] (37) at (-3.5, -1.75) {};
		\node [style=empty-vertex] (38) at (-0.5, -1.5) {};
		\node [style=empty-vertex] (39) at (-1, -1.75) {};
		\node [style=empty-vertex] (40) at (0, -1.75) {};
		\node [style=empty-vertex] (41) at (-4.75, -2.25) {$\color{blue}{-2}$};
		\node [style=empty-vertex] (42) at (0.25, -2.25) {$\color{red}{2}$};
		\node [style={black_dot}] (43) at (2.5, -2.25) {$v_0$};
		\node [style={black_dot}] (44) at (3.5, -3.25) {$v_1$};
		\node [style={black_dot}] (45) at (5, -3.25) {$v_2$};
		\node [style={black_dot}] (46) at (6, -2.25) {$v_3$};
		\node [style={black_dot}] (47) at (3.5, -4.5) {$v_4$};
		\node [style={black_dot}] (48) at (5, -4.5) {$v_5$};
		\node [style=empty-vertex] (49) at (2.5, -1.5) {};
		\node [style=empty-vertex] (50) at (2, -1.75) {};
		\node [style=empty-vertex] (51) at (3, -1.75) {};
		\node [style=empty-vertex] (52) at (6, -1.5) {};
		\node [style=empty-vertex] (53) at (5.5, -1.75) {};
		\node [style=empty-vertex] (54) at (6.5, -1.75) {};
		\node [style=empty-vertex] (55) at (1.75, -2.25) {$\color{red}{-2}$};
		\node [style=empty-vertex] (56) at (6.75, -2.25) {$\color{blue}{2}$};
	\end{pgfonlayer}
	\begin{pgfonlayer}{edgelayer}
		\draw [style={dashed_edge}] (0) to (7);
		\draw [style={dashed_edge}] (0) to (6);
		\draw [style={dashed_edge}] (0) to (8);
		\draw [style={dashed_edge}] (3) to (10);
		\draw [style={dashed_edge}] (3) to (9);
		\draw [style={dashed_edge}] (3) to (11);
		\draw [style={dashed_edge}] (15) to (22);
		\draw [style={dashed_edge}] (15) to (21);
		\draw [style={dashed_edge}] (15) to (23);
		\draw [style={dashed_edge}] (18) to (25);
		\draw [style={dashed_edge}] (18) to (24);
		\draw [style={dashed_edge}] (18) to (26);
		\draw [style={dashed_edge}] (29) to (36);
		\draw [style={dashed_edge}] (29) to (35);
		\draw [style={dashed_edge}] (29) to (37);
		\draw [style={dashed_edge}] (32) to (39);
		\draw [style={dashed_edge}] (32) to (38);
		\draw [style={dashed_edge}] (32) to (40);
		\draw [style={dashed_edge}] (43) to (50);
		\draw [style={dashed_edge}] (43) to (49);
		\draw [style={dashed_edge}] (43) to (51);
		\draw [style={dashed_edge}] (46) to (53);
		\draw [style={dashed_edge}] (46) to (52);
		\draw [style={dashed_edge}] (46) to (54);
		\draw [style={blue_arc}] (1) to (0);
		\draw [style={blue_arc}] (18) to (17);
		\draw [style={blue_arc}] (29) to (30);
		\draw [style={blue_arc}] (45) to (46);
		\draw [style={red_arc}] (43) to (44);
		\draw [style={red_arc}] (31) to (32);
		\draw [style={red_arc}] (16) to (15);
		\draw [style={red_arc}] (3) to (2);
		\draw [style=arrow] (4) to (1);
		\draw [style=arrow] (1) to (2);
		\draw [style=arrow] (2) to (5);
		\draw [style=arrow] (19) to (16);
		\draw [style=arrow] (16) to (17);
		\draw [style=arrow] (17) to (20);
		\draw [style=arrow] (30) to (33);
		\draw [style=arrow] (31) to (30);
		\draw [style=arrow] (34) to (31);
		\draw [style=arrow] (45) to (44);
		\draw [style=arrow] (44) to (47);
		\draw [style=arrow] (48) to (45);
	\end{pgfonlayer}
\end{tikzpicture}
        \caption{Four cases for which a solution was not found using a computer.}
        \label{fig:four_unsolved_cases_by_computer}
    \end{figure}
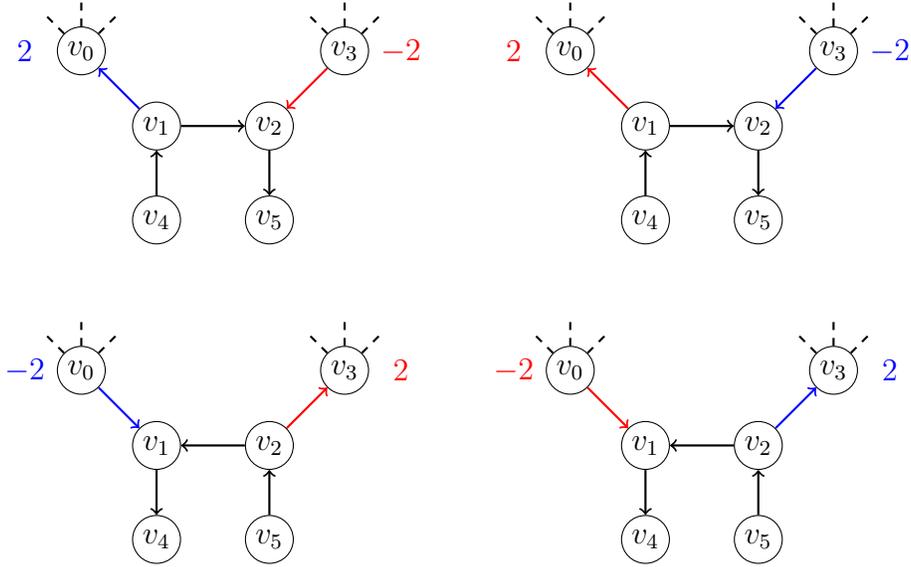

    Suppose that $C$ has a length of 4 or more.
    If $d(v) = 2$ for every $v \in V(C) \setminus \{x\}$, take a digraph $D_4$ obtained from $D$ by removing all vertices of $C$ except $x$ and its two neighbors.
    The removed part is an orientation of a path of length at least 2.
    Hence, the removed part of $D$ can be decomposed into both even or odd numbers of strong locally irregular digraphs -- each consisting of one or two consecutive arcs. 
    Alternating blue and red colors on incident digraphs from such a decomposition results in a coloring of the removed digraph where the first and the last arcs are colored as needed -- with colors different from the colors of arcs they are attached to $x$. 
    Such a coloring combined with a strong locally irregular coloring of $D_4$ yields a strong locally irregular coloring with two colors of $D$, a contradiction. Hence:
    \begin{claim}
        If the length of $C$ is at least $4$ then $d(v_1) \geq 3$.
    \end{claim}

    If the length of $C$ is at least 4, at least one of the vertices $v_0$ and $v_3$ is different from $x$.
    Suppose that $v_0 \neq x$.
    The choice of $v_2$ (a neighbor $v_1$ on $C$ different from $x$ with the maximum degree) implies $d(v_0) = 2$.
    In this case, denote by $D_5$ a digraph obtained from $D$ by removing $v_1$ and its pendant neighbor.
    A strong locally irregular coloring of $D_5$ can be extended to a strong locally irregular coloring with two colors of $D$ in the following way:
    If the arcs between $v_0$ and its neighbor in $D_5$, and between $v_2$ and its neighbor in $D_5$ are colored the same, w.l.o.g. blue, color the pendant arc incident to $v_1$ blue, and the remaining two arcs red. If these two arcs of $D_5$ are colored with different colors, w.l.o.g. the arc incident to $v_0$ is blue and the arc incident to $v_2$ is red, color the arc between $v_1$ and $v_2$ blue, and the remaining two arcs red.
    For an overview of these cases see Figure~\ref{fig:cases_3-2}.
    On the other hand, if $v_0 = x$, then $v_3 \neq x$ and from the assumption that $v_1$ has the maximum degree among all vertices from $V(C) \setminus \{x\}$ we get that $d(v_2) \in \{2,3\}$. 
    Considering obtainable configurations, using a Python program, see 'S5 Proof of Claim 5' in~\cite{supplementary_material}, we complete the proof of the following:
    \begin{claim}\label{claim5}
        If the length of $C$ is at least $4$ then $(d(v_1), d(v_2)) \neq (3,2)$.
    \end{claim}
    This completes the proof of (i).

    \begin{figure}[h]
        \centering
        \begin{tikzpicture}[scale=1]
	\begin{pgfonlayer}{nodelayer}
		\node [style={black_dot}] (0) at (-3, 2) {$v_0$};
		\node [style={black_dot}] (1) at (-2, 1) {$v_1$};
		\node [style={black_dot}] (2) at (-0.5, 1) {$v_2$};
		\node [style={black_dot}] (3) at (1, 1) {$v_3$};
		\node [style={black_dot}] (4) at (2, 2) {$v_4$};
		\node [style={black_dot}] (5) at (-2, -0.25) {$v_5$};
		\node [style=empty-vertex] (6) at (-3, 2.75) {};
		\node [style=empty-vertex] (7) at (-3.5, 2.5) {};
		\node [style=empty-vertex] (8) at (-2.5, 2.5) {};
		\node [style=empty-vertex] (9) at (1.5, 2.5) {};
		\node [style=empty-vertex] (10) at (2, 2.75) {};
		\node [style=empty-vertex] (11) at (2.5, 2.5) {};
		\node [style={black_dot}] (12) at (4.5, 2) {$v_0$};
		\node [style={black_dot}] (13) at (5.5, 1) {$v_1$};
		\node [style={black_dot}] (14) at (7, 1) {$v_2$};
		\node [style={black_dot}] (15) at (8.5, 1) {$v_3$};
		\node [style={black_dot}] (16) at (9.5, 2) {$v_4$};
		\node [style={black_dot}] (17) at (5.5, -0.25) {$v_5$};
		\node [style=empty-vertex] (18) at (4.5, 2.75) {};
		\node [style=empty-vertex] (19) at (4, 2.5) {};
		\node [style=empty-vertex] (20) at (5, 2.5) {};
		\node [style=empty-vertex] (21) at (9, 2.5) {};
		\node [style=empty-vertex] (22) at (9.5, 2.75) {};
		\node [style=empty-vertex] (23) at (10, 2.5) {};
		\node [style={black_dot}] (24) at (8.5, -0.25) {$v_6$};
	\end{pgfonlayer}
	\begin{pgfonlayer}{edgelayer}
		\draw [style={dashed_edge}] (7) to (0);
		\draw [style={dashed_edge}] (0) to (6);
		\draw [style={dashed_edge}] (0) to (8);
		\draw [style={dashed_edge}] (4) to (9);
		\draw [style={dashed_edge}] (4) to (10);
		\draw [style={dashed_edge}] (4) to (11);
		\draw (0) to (1);
		\draw (1) to (2);
		\draw (2) to (3);
		\draw (3) to (4);
		\draw (1) to (5);
		\draw [style={dashed_edge}] (19) to (12);
		\draw [style={dashed_edge}] (12) to (18);
		\draw [style={dashed_edge}] (12) to (20);
		\draw [style={dashed_edge}] (16) to (21);
		\draw [style={dashed_edge}] (16) to (22);
		\draw [style={dashed_edge}] (16) to (23);
		\draw (12) to (13);
		\draw (13) to (14);
		\draw (14) to (15);
		\draw (15) to (16);
		\draw (13) to (17);
		\draw (15) to (24);
	\end{pgfonlayer}
\end{tikzpicture}
        \caption{Cases considered in Claim~\ref{claim5}.}
        \label{fig:cases_3-2}
    \end{figure}
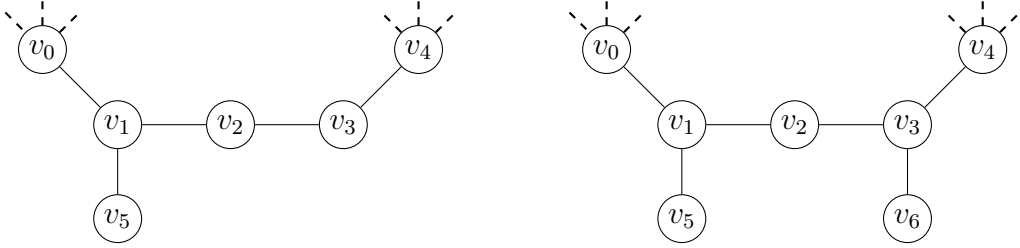

    Hence, every pendant cycle of $G$ has length 3. 
    Take a pendant cycle $C$ such that $G - V(C)$ has at most one component that contains cycles (such a cycle always exists in a cactus different from a tree).
    Denote the vertices on $C$ by $x, v_1, v_2$ such that $v_1$ and $v_2$ do not lie on any other cycle in $G$, and $d(v_1) \geq d(v_2)$.
    Claim~\ref{claim1} and Claim~\ref{claim2} (no assumption on the length of $C$ was used to prove these claims) yields that $d(v_2) = 2$ and either $d(v_1) = 2$ or $d(v_1) = 3$.
    
    To prove (ii), suppose to the contrary that $x$ does not lie on any other cycle of $G$.
    This in particular means that there is a neighbor $y$ of $x$ that lies in a component of $G-x$ with cycles, and $x$ is incident to $d(x) - 3$ pendant arcs (since every pendant tree of $G$ is a star, see Observation~\ref{obs_min_bad_cactus}).
    Note that $d(x)$ can be upper bounded in a similar way that we used to show $d(v_1) \leq 7$ (i.e., $k \leq 5$) in the proof of (i).
    If $d(x) \geq 10$, remove $d(x) - 3$ pendant neighbors of $x$ and color the obtained digraph in a strong locally irregular way with 2 colors.
    Any coloring of the pendant arcs incident to $x$ changes both color degrees of $x$ in a way different from the others, and there are at most seven possibilities for the balanced color degrees of $x$ that would violate the strong locally irregular condition: three for the arcs incident to $x$ in the digraph obtained from $D$ by removing the pendant arcs incident to $x$, and four on the pendant arcs incident to $x$ ($-1$ or $1$ in blue or red). 
    Hence, if $d(x) - 3 \geq 7$ there are at least $8$ ways of colorings of arcs incident to $x$ (see Observation~\ref{obs_oriented_star_1}), and at least one of them does not create any conflict with the strong locally irregular condition.
    Thus, we may assume $d(x) \leq 9$, i.e., $x$ is incident to at most $6$ pendant arcs.

    Let $D'$ be a digraph obtained from $D$ by removing $v_1$ and its pendant neighbor (if present), $v_2$, and all pendant neighbors of $x$. Using a python program, see 'S6 Proof of (ii)' in~\cite{supplementary_material}, we show that a strong locally irregular coloring of $D'$ can be extended to a strong locally irregular coloring  with two colors of $D$ for every possible orientation of edges incident to $x$, $v_1$, and $v_2$. This completes the proof of (ii).
     \end{proof}

As immediate consequences of Lemma~\ref{lemma_counterexample_cactus} we get:
\begin{tw}
    If $D$ is an orientation of a cactus graph $G$ without cycles of length three then $\operatorname{slir}(D) \leq 2$.
\end{tw}
\begin{proof}
    This follows from 1) of Lemma~\ref{lemma_counterexample_cactus}.
\end{proof}
\begin{tw}
    If $D$ is an orientation of a cactus graph $G$ with vertex-disjoint cycles then $\operatorname{slir}(D) \leq 2$.
\end{tw}
\begin{proof}
    Let $G$ be a cactus with the minimum number of vertices and edges which has an orientation $D$ with $\operatorname{slir}(D) \geq 3$.
    Take a pendant cycle $C$ of $G$ such that at most one component of $G-V(C)$ contains cycles. 
    From 1) of Lemma~\ref{lemma_counterexample_cactus} we have that the length of $C$ is three.
    From 2) we get that there is a vertex of $C$ which lies on multiple cycles in $G$, contradicting the assumption that cycles of $G$ are vertex-disjoint.
\end{proof}

\section{Concluding remarks}

In this paper, we introduced and investigated two new versions (weak and strong) of the local irregularity of digraphs.
We discussed the relation of weak and strong locally irregular colorings to other locally irregular colorings of digraphs, in particular to $(+,+)$-locally irregular colorings introduced in~\cite{Bensmail Renault}.

While $(+,+)$-locally irregular coloring is a special case of weak locally irregular coloring (hence, $\mathrm{lir}^{(+,+)}(D) \leq k$ implies $\mathrm{lir}(D) \leq k$), it is not true that every $(+,+)$-locally irregular coloring is a strong locally irregular coloring or vice versa.
Moreover, the weak local irregularity of an orientation of a graph $G$ is implied by the local irregularity of $G$ itself.
In the case of the strong local irregularity, we showed its independence from the local irregularity of simple graphs.
Hence, the strong local irregularity, compared to the weak one, better utilizes core properties of digraphs, while, on the other hand, the results on weak locally irregular colorings and decompositions of oriented graphs in particular, could give some information on the locally irregular colorings of their skeletons.

From these notes the following problem arises naturally:
\begin{problem}
    Provide a general constant upper bound on $\operatorname{slir}(D)$.
\end{problem}
This problem can be reduced to providing a general constant upper bound for oriented graphs.
In a digraph $D$, take a subdigraph $D'$ induced on arcs going in both directions between pairs of vertices.
Then take a subdigraph $D''$ of $D'$ that is strong locally irregular and contains exactly one of each pair of arcs going in opposite direction between two vertices (such a digraph always exists, see Theorem~\ref{Borowiecki2}).
The subdigraph $D$ induced on the arcs $E(D) \setminus E(D'')$ is then an orientation of the skeleton of $D$.
Hence, a general constant upper bound on the strong locally irregular chromatic index for oriented graphs yields a constant upper bound on it in the case of general digraphs. 
In particular, we have the following analog of Lemma~\ref{2}:
\begin{lem}
    Let $D$ be a digraph with a skeleton $G$. If the strong locally irregular chromatic index of every orientation of $G$ is at most $k$ then $\operatorname{slir}(D) \leq k+1$.
\end{lem} 

In the case of orientations of cacti, one could possibly continue with the approach we provided, consider and solve possible cases (after lowering the number of the interesting ones to a finite number). 
However, such an approach would probably not yield anything new that would be of much interest from the combinatorial perspective, except hopefully proving Conjecture~\ref{main1} for such digraphs.
Hence, we suggest a different approach, which could be based on the proof of Theorem~\ref{thm_slir_unicyclic}. 
In that proof a cycle is split, and we take a specific strong locally irregular coloring of the obtained digraph, whose skeleton is a tree.
The fact that we have information on how the arcs incident to each vertex are colored, is key.
Such a coloring yields a strong locally irregular coloring of the original oriented graph, however, the property of the coloring is not preserved.
Hence, a simple induction on the number of cycles of a cactus is not possible. 
But we believe that there might be a way to carefully describe the order in which cycles are split in a cactus, such that in each step we have control on the obtained color degrees of vertices.

Nevertheless, after fully proving that Conjecture~\ref{main1} holds for orientations of all cacti, a natural step would be to consider orientations of outerplanar graphs, or more specifically subcubic outerplanar graphs, and digraphs whose skeletons are cacti.

In the case of weak locally irregular colorings, as the primary goal we see lowering the general upper bound that is currently obtained as direct consequence of the general upper bound proved for $(+,+)$-locally irregular colorings.
For example, proving that Conjecture~\ref{main} is true for every acyclic digraph yields a general upper bound of four (as every digraph can be decomposed into two acyclic digraphs). 
From another point of view, results on locally irregular colorings of multigraphs where the multiplicity of each edge is at most two could be naturally used to bound $\operatorname{lir}$ in the case of corresponding digraphs (a multiedge is replaced by arcs going in opposite direction).

\end{document}